\documentclass{article}
\usepackage[utf8]{inputenc}

\usepackage{mathrsfs}
\usepackage{amsmath}
\usepackage{amssymb}
\usepackage{amsthm,bm}
\usepackage[colorlinks,linkcolor=red,anchorcolor=blue,citecolor=green]{hyperref}
\usepackage{graphics,graphicx}
\usepackage{color}
\usepackage{enumerate}
\usepackage[title]{appendix}

\usepackage{subfigure}

\usepackage{caption}
\allowdisplaybreaks

\usepackage{enumitem}
\setlist[enumerate,1]{label=(\alph*), font=\normalfont\bfseries}

\usepackage{tikz}
\usepackage{indentfirst}

\newtheorem{Thm}{Theorem}[section]
\newtheorem{Def}[Thm]{Definition}
\newtheorem{Lemma}[Thm]{Lemma}
\newtheorem{Coro}[Thm]{Corollary}

\newtheorem{Prop}[Thm]{Proposition}
\newtheorem{Remark}[Thm]{Remark}

\newtheorem{Claim}[Thm]{Claim}
\newtheorem{Fact}[Thm]{Fact}
\newtheorem{Que}[Thm]{Question}
\title{Edge-averaging dynamics on finite graphs:  moment dependence}
\author{Junchi Zuo}
\date{}

\usepackage[square,numbers]{natbib}

\begin{document}

\maketitle
	\begin{abstract}

We study the edge-averaging process on a finite, connected graph $G = (V, E)$. 
Initially, the vertices in $V$ are endowed with i.i.d.\ real-valued opinions $(f_0(v))_{v \in V}$. Edges are activated    according to i.i.d.\ Poisson clocks of rate $1$; when an edge is activated, the opinions at its  endpoints are replaced by their average. Let $f_t(v)$ denote the opinion at  $v$ at time $t$.
Define the $\epsilon$-convergence time $\tau_\epsilon$ as the first time when the maximum and the minimum of $f_t$ differ by at most $\epsilon$. 

It is known that if the initial opinions $(f_0(v))_{v \in V}$ are  bounded in $L^\infty$, then $\mathbb{E}(\tau_\epsilon)$ is at most $C_\epsilon \log^2 n$ for $\epsilon \in (0, 1]$.

We assume instead that the $L^p$ norm of $f_0(v)$ is at most $1$ for every $v \in V$. For fixed $\epsilon \in (0, 1]$, and show that $\mathbb{E}(\tau_\epsilon) = \Tilde{O}(n^{\beta_p})$ up to logarithmic terms, where $\beta_p := \max(3 - p, 2/p)$.  Moreover, this power law is tight on cycle graphs. 


	\end{abstract}

\section{Introduction}

\subsection{Problem setting}
Motivated by applications in sensor networks and social networks, we study the {\bf edge-averaging process} on a graph $G = (V, E)$. In this model, each vertex $v \in V$ is assigned an initial opinion $f_0(v) \in \mathbb{R}$ and the opinion evolves over positive real time $t \ge 0$. We denote the opinion of vertex $v \in V$ at time $t \ge 0$ by $f_t(v)$.  

The opinions evolve over time as follows. Each edge is equipped with an independent unit-rate Poisson clock. When the clock of an edge $ \{v, w \} \in E$ rings at time $t \ge 0$, both $\{v, w \}$ update their opinions to be the average of the two. That is, 
\begin{align}
    f_t(v) = f_t(w) := \frac{1}{2}(f_{t-}(v) + f_{t-}(w)) \, , 
\end{align}
while the opinions of vertices other than $\{v, w \}$ do not change.

\begin{Fact} \label{fact: almost sure convergence}
    For any finite connected graph $G = (V, E)$ and any initial profile $\{f_0(v) \}_{v \in V}$, we have $\lim_{t \to \infty} f_t(v) = L \,$ for any $v \in V$ almost surely, where $L := \frac{1}{\lvert V \rvert} \sum_{v \in V} f_0(v)$.  
\end{Fact}

Fact \ref{fact: almost sure convergence} is proved in \cite{boyd2006randomized} by Boyd, Ghosh, Prabhakar and Shah. 

\begin{Que} \label{que: how fast it converges}
    How fast does the dynamics converge for finite, connected graphs?   
\end{Que}

We denote the oscillation of a function $f: V \to \mathbb{R}$ by $\mathrm{osc} \ f$, whose precise definition is stated as below
\begin{align}
    \mathrm{osc} \ f := \sup_{v \in V} f(v) - \inf_{v \in V} f(v) \, .
\end{align}

In order to quantify the speed of convergence, we define the {\bf$\epsilon$-consensus time} by 
\begin{align}
    \tau_\epsilon := \inf \{t \ge 0: \mathrm{osc} \ f_t \le \epsilon \} \, .
\end{align}

\subsection{Main results: bounds for expected $\epsilon$-consensus time}


Elboim, Peres and Peretz considered the case that $\{f_0(v)\}_{v \in V}$ are i.i.d.\ random variables taking values in $[-1, 1]$ and proved that $\mathbb{E}(\tau_\epsilon) = O(\epsilon^{-4} \log^2(n/\epsilon))$ in \cite{gantert2024averagingprocessinfinitegraphs}. They also showed this bound is tight when $G$ is an $n$-cycle and $f_0(v) \sim \mathrm{Uniform}\{-1, 1\}$. 


However, we consider the setting that $\{ f_0(v)\}_{v \in V}$ are random variables whose $L^p$ norm is bounded by $1$ for $p \in [1, \infty)$. Theorem \ref{thm: upper bound for L^1} and Theorem \ref{thm: upper bound for p > 1} provide upper bounds for $\mathbb{E}(\tau_\epsilon)$. Theorem \ref{thm: upper bound for L^1} focuses on the $p = 1$ case, while Theorem \ref{thm: upper bound for p > 1} addresses the case $p > 1$. We emphasize that we impose the i.i.d.\ assumption on the initial profile $\{f_0(v)\}_{v \in V}$ in Theorem \ref{thm: upper bound for p > 1}, but do not need any extra assumption in Theorem \ref{thm: upper bound for L^1}. 

\begin{Thm} \label{thm: upper bound for L^1}
    There exists a constant $C > 0$ such that for every finite, connected graph $G = (V, E)$ with $\lvert V \rvert = n \ge 2$,  the following statement holds: for any $\epsilon \in (0, 1]$ and any random initial profile $\{f_0(v)\}_{v \in V}$ such that $\mathbb{E}(\lvert f_0(v) \rvert) \le 1$  for all $v \in V$, we have
    $$ \mathbb{E}(\tau_\epsilon) \le Cn^2 \log(2/\epsilon)\, .$$
\end{Thm}

\begin{Thm} \label{thm: upper bound for p > 1}
    There exists a constant $C > 0$ such that for any $p \in (1, \infty)$, any finite, connected graph $G = (V, E)$ with $\lvert V \rvert = n \ge 2$, the following statement holds: for any $\epsilon \in (Cn^{\frac{1}{p} - 1} \log n, 1]$, and any random variable $X$ with $\mathbb{E}(\lvert X \rvert^p) \le 1$, the consensus time $\tau_\epsilon$ for $G$ and i.i.d.\ initial profile $f_0: V \to \mathbb{R}$ with $f_{0}(v) \stackrel{d}\sim X$, we have
    \begin{align}
        \mathbb{E}(\tau_\epsilon) \le \begin{cases}
            Cn^{3 - p} \epsilon^{-p} \log^2(n/\epsilon) & p \in (1, 2) \, , \\ C\left(n^{2/p}\epsilon^{-2} + \epsilon^{-4} \right) \log^2(n/\epsilon) & p \in [2, \infty) \, .
        \end{cases}
    \end{align}

\end{Thm}

Theorem \ref{thm: lower bound} below shows that the upper bounds in Theorem \ref{thm: upper bound for L^1} and Theorem \ref{thm: upper bound for p > 1} are tight up to some logarithmic factors. 

\begin{Thm} \label{thm: lower bound}
    There exists a constant $C > 0$ such that for any $p \in [1, \infty)$ and positive integer $n > C$, there exists a random variable $X$ with $\mathbb{E}(X) = 0$ and $\mathbb{E}(\lvert X \rvert^p) \le 1$ such that the following statement holds: for any $\epsilon \in [n^{\frac{1}{p} - 1}, 1]$, the consensus time $\tau_\epsilon$ for the $n$-cycle graph $G = (V, E)$ and the i.i.d.\ initial profile  $\{f_0(v)\}_{v \in V}$ with $f_0(v) \stackrel{d} \sim X$ satisfies 
    \begin{align}
        \mathbb{E}(\tau_\epsilon) \ge \begin{cases}
            C^{-1} n^{3 - p} \epsilon^{-p} & p \in [1, 2) \, , \\ C^{-1} \left(n^{2/p} \epsilon^{-2} + \epsilon^{-4} \right)& p \in [2, \infty)\, .
        \end{cases}
    \end{align}
\end{Thm}

\begin{Remark}
    For $p = 1$, Theorem \ref{thm: lower bound} implies that there exists a constant $C > 0$ such that $\mathbb{E}(\tau_\epsilon) \ge C^{-1}n^2$ for any $\epsilon \in (0, 1]$. This follows since $\mathbb{E}(\tau_\epsilon) \ge \mathbb{E}(\tau_1) \ge C^{-1}n^2$. 
\end{Remark}

\subsection{Related works}

Beyond Fact \ref{fact: almost sure convergence}, a more difficult question is whether for infinite, connected graphs $G = (V, E)$, the limit $ \lim_{t \to \infty}  f_t(v)$  exists almost surely for each $v \in V$.  Gantert and Vilkas~\cite{gantert2024averagingprocessinfinitegraphs} proved that for i.i.d.\ initial profile $\{f_0(v)\}_{v \in V}$ with finite second moment, $f_t(v)$ converges to $\mathbb{E}(f_0(v))$ in $L^2$ for any single vertex $v \in V$. Regarding almost surely convergence, Elboim, Peres and Peretz~\cite{elboim2025edgeaveragingprocessgraphsrandom} considered the case that $\{f_0(v) \}_{v \in V}$ are i.i.d.\ and have finite $L^{4 + \epsilon}$ norm. Under this stronger condition, they proved    that $f_t(v)$ converges almost surely to $\mathbb{E}(f_0(v))$ as $t \to \infty$  for every vertex $v \in V$. However, it remains an  open question  which weaker condition suffices for the almost sure convergence of $f_t(v)$ on infinite connected graphs.

Several authors have studied the speed of convergence of the edge-averaging dynamics on general finite graphs. Boyd, Ghosh, Prabhakar and Shah \cite{boyd2006randomized} studied the $L^2$ consensus time of the edge-averaging dynamics on general graphs and related it to the second eigenvalue of a weighted adjacency matrix. Aldous and Lanoue \cite{aldous2012lecture} also studied the convergence of the model and related it to a coupled Markov chain. Elboim, Peres and Peretz \cite{elboim2025edgeaveragingprocessgraphsrandom} considered the expected $\epsilon$-consensus time on general graphs, where they assume that the initial profile $\{ f_0(v)\}_{v \in V}$ are i.i.d.\ and bounded in $L^\infty$. In this setting, they showed that $\mathbb{E}(\tau_\epsilon) = O(\epsilon^{-4} \log^2(n/\epsilon))$. Moreover, the paper \cite{gollin2025sharingteagraph} considered a general setting beyond the edge-averaging dynamics. They studied the setting where an averaging operation on a edge can be performed for arbitrary times and in arbitrary order.   

Several  works focus  on the behavior of the edge-averaging process on specific graphs. Chatterjee, Diaconis, Sly, and Zhang \cite{zbMATH07496855} established a cutoff phenomenon for the discrete time edge-averaging process on complete graphs. Caputo, Quattropani and Sau \cite{arXiv:2603.00705} studied the cutoff of edge-averaging process on random regular graphs and found a surprising dependence on the degree. 

While the paper \cite{elboim2025edgeaveragingprocessgraphsrandom} assumes that the initial profile $\{f_0(v)\}_{v \in V}$ are i.i.d.\ with finite $L^\infty$ norms, our work only assumes the $L^p$ norm of $f_0(v)$ is bounded by $1$. This generalization can be applied to the study of heavy-tailed distributions that may not have all moments. Many basic and essential properties of heavy-tailed distributions can be found in the book \cite{nair2022fundamentals}.  

\section{Fragmentation process on finite graphs}


\begin{Def}[General Fragmentation process]
    Let $\mu_0: V \to \mathbb{R}_{\ge 0}$ be any initial mass distribution on a graph $G = (V, E)$ with $\sum_{v \in V} \mu_0(v) = 1$. We define the {\bf fragmentation process} $\{ \mu_t(v)\}_{t \ge 0, v \in V}$ inductively as follows: When the clock of an edge $\{v, w \} \in E$ rings at time $t \ge 0$, we let $\mu_t(v) = \mu_t(w) := \frac{1}{2}(\mu_{t^-}(v) + \mu_{t^-}(w))$ while the other value of $\mu_t$ does not change. 
\end{Def}


\begin{Def}\label{def: M_t and M^*_t}
    Let $G = (V, E)$ be a finite graph. For any fixed $v \in V$, we define $M_t(v, \cdot)$ by the fragmentation process with respect to initial mass distribution $\delta_{v}$. In other words, we have $M_0(v, w) = 1_{v = w}$ at time $0$. When the clock of an edge $\{w_1, w_2 \} \in E$ rings at time $t$, we let
    \begin{align}
        M_t(v, w_1) = M_t(v, w_2) := \frac{1}{2}\left(M_{t^-}(v, w_1) + M_{t^-}(v, w_2) \right) \quad\forall v \in V \, ,
    \end{align}
    
    while keeping the value of $M_t(v, w)$ for $w \neq \{w_1, w_2 \}$ fixed. 
    
    Moreover, we define the shifted matrix $M^*_t := M_t - \frac{1}{n}J$, where $J$ is the all-ones square matrix indexed by $V$. 
\end{Def}




We introduce the following notations. Denote the random sequence $(U_i, \chi_i)_{i \ge 1} \in \mathbb{R}_{\ge 0} \times E$ to be the clock rings for all edges in $E$, ordered by time, and define $U_0 := 0$. In detail, we have $0 = U_0 < U_1 < U_2 < \cdots $ almost surely, and the clock of edge $\chi_i$ rings at time $U_i$. 

For any $v \in V$, we denote $\mathbf{e}_v$ to be the column vector with only entries $1$ on $v \in V$. In other words, we have $$ \mathbf{e}_v := \begin{pmatrix} 0 \ 0 \cdots 1 \ 0 \cdots 0 \end{pmatrix}^T ,$$ 
where the only nonzero entry represents vertex $v$. For any edge $e = \{v, w \} \in E$, we define $$W_e := I - \frac{1}{2}(\mathbf{e}_v - \mathbf{e}_w)(\mathbf{e}_v - \mathbf{e}_w)^T \, .$$ 

We make the convention that any measure on $V$ is regarded as a row vector, while any function on $V$ is regarded as a column vector. For any matrix $M \in \mathbb{R}^{V \times V}$, we define the $\ell^\infty$ norm of $M$ by $\lVert M \rVert_\infty := \sup_{v, w \in V} M(v, w)$. 

We can represent the matrix $\{M_t \}_{t \ge 0}$ by the multiplication of a series of matrices $W_{\chi_i}$, as shown in Claim \ref{cl: matrix multiplication}. 

\begin{Claim} \label{cl: matrix multiplication}
    The following properties hold almost surely.
    \begin{enumerate}
        \item For any $t \ge 0$, we have $M_t = \prod_{U_i \le t} W_{\chi_i}$. 
        \item For any $t \ge 0$, the matrix $M_t$ is doubly stochastic, i.e.\ $M_t$ has non-negative entries and the row sums and column sums equal $1$. 
        \item Let $\{\mu_t(v)\}_{t \ge 0, v \in V}$ be a fragmentation process on a graph $G = (V, E)$. Then we have $\mu_t = \mu_0 M_t$ for any $t \ge 0$. 
        \item Let $\{f_t\}_{t \ge 0, v \in V}$ be an edge-averaging process on a graph $G = (V, E)$. Then we have $f_t = M_t^T f_0$ for any $t \ge 0$. 
    \end{enumerate}
\end{Claim}
\begin{proof}
    Since $M_t$, $\mu_t$ and $f_t$ remains constant during time $(U_{i - 1}, U_i]$, it suffices to deal with the case where $t = U_k$, $k \in \mathbb{N}$. Parts (a), (c) and (d) can be easily proved by induction on $k$. Since multiplication of two doubly stochastic matrices is doubly stochastic, Part (b) also holds. 
\end{proof}

Elboim, Peres and Peretz\cite{elboim2025edgeaveragingprocessgraphsrandom}  analyzed the fragmentation process on finite connected graphs and proved the following theorem below.

\begin{Thm}[\cite{elboim2025edgeaveragingprocessgraphsrandom}, Lemma 2.2]  \label{thm: Q_t estimate, quoted}
    For any finite, connected graph $G = (V, E)$ with $\lvert V \rvert = n$, any initial mass distribution $\{\mu_0(v)\}_{v \in V}$, and any $t \ge 0$, the fragmentation process $\{\mu_t(v)\}_{t \ge 0, v \in V}$ satisfies $$\mathbb{P}(Q_t \ge 6t_*^{-1/2}) \le \exp(-t_*/30) \, ,$$
    where $Q_t := \sum_{v \in V} \mu_t(v)^2$ and $t_* = \min(t, n^2)$. 
\end{Thm}

The following Corollary follows directly from Theorem \ref{thm: Q_t estimate, quoted}. 

\begin{Coro} \label{coro: estimate of L2 norm}
    There exists a universal constant $C > 0$ such that for any finite, connected graph $G = (V, E)$ with $\lvert V \rvert = n$ and any $t \in (0, n^2]$, we have
    \begin{align}
        \mathbb{P}\Bigl(\Big\{\forall v \in V, \sum_{w \in V} M_t^2(v, w) \le 6 t^{-1/2}\Big\} \text{ and } \Big\{\forall w \in V, \sum_{v \in V} M_t^2(v, w) \le 6t^{-1/2}\Big\}\Bigr) \ge 1 - 2n\exp(-t/30) \, .
    \end{align}
\end{Coro}
\begin{proof}
    Fix $v \in V$. Recall that by Definition \ref{def: M_t and M^*_t}, $M_t(v, \cdot)$ is a fragmentation process with respect to initial mass $\delta_{v}$. Hence, by Theorem \ref{thm: Q_t estimate, quoted}, we have 
    \begin{align}
        \mathbb{P}\Bigl(\sum_{w \in V} M_t^2(v, w) \ge 6t^{-1/2}\Bigr) \le \exp(-t/30) \, .
    \end{align}


    By Claim \ref{cl: matrix multiplication}, we have $M_t = \prod_{U_i \le t} W_{\chi_i}$. Note that $W_e$ is symmetric for any $e \in E$, and the Poisson clock of the reversed order during time period $[0, t]$ has the same distribution as the original one. Therefore, $M_t$ and $M_t^T$ have the same distribution. This implies that for any $w \in V$, we have 
    \begin{align}
        \mathbb{P}\Bigl(\sum_{v \in V} M_t^2(v, w) \ge 6t^{-1/2}\Bigr) \le \exp(-t/30) \, .
    \end{align}
    
    Combining with a simple union bound, we conclude the proof of the Corollary. 
\end{proof}


Corollary \ref{cor: estimate for uniform norm} below gives a tail inequality for $\lVert M_t \rVert_\infty$ for $t \in (0, n^2]$, by applying Corollary \ref{coro: estimate of L2 norm}.

\begin{Coro} \label{cor: estimate for uniform norm}
   For any finite, connected graph $G = (V, E)$ with $\lvert V \rvert = n \ge 2$, and any $t \in (0, n^2]$, we have
    \begin{align}
        \mathbb{P}\bigl(\lVert M_t \rVert_\infty \ge 10 t^{-1/2}\bigr) \le 4n \exp(-t/60) \, ,
    \end{align}
    where $M_t$ is defined in Definition \ref{def: M_t and M^*_t}. 
\end{Coro}
\begin{proof}



    We define the matrix $\Tilde{M}_{t/2} := \prod_{t/2 < U_i \le t} W_{\chi_i}$. By Claim \ref{cl: matrix multiplication}, we have $M_t = M_{t/2} \Tilde{M}_{t/2}$. Since the Poisson clock during time $[t/2, t]$ has the same distribution as Poisson clock during time $[0, t/2]$, $M_{t/2}$ and $\Tilde{M}_{t/2}$ have the same distribution. 

    By the Cauchy-Schwarz inequality, we have
    \begin{align} \label{eq: usage of Cauchy inequality}
        M_t(v, w) \le \sqrt{\Big(\sum_{u \in V} M_{t/2}(u, w)^2\Big)\Big(\sum_{u \in V} \Tilde{M}_{t/2}(v, u)^2\Big)} \quad \forall v, w \in V \, .
    \end{align}
    
    By Corollary \ref{coro: estimate of L2 norm}, we have
    \begin{align}
        \mathbb{P}\Bigl(\Big\{\forall v \in V, \sum_{w \in V} M_{t/2}^2(v, w) \le \frac{6\sqrt 2}{\sqrt t} \Big\} \text{ and } \Big\{\forall v \in V, \sum_{w \in V} M_{t/2}^2(w, v) \le \frac{6 \sqrt 2}{\sqrt t} \Big\}\Bigr) \ge 1 - 2n\exp(-t/60) \, .
    \end{align}

    Since the matrix $\Tilde{m}_{t/2}(\cdot, \cdot)$ has the same law as $m_{t/2}(\cdot, \cdot)$, we also have
    \begin{align}
        \mathbb{P}\Bigl(\Big\{\forall v \in V, \sum_{w \in V} \Tilde{M}_{t/2}^2(v, w) \le \frac{6\sqrt 2}{\sqrt t} \Big\} \text{ and } \Big\{\sum_{w \in V} \Tilde{M}_{t/2}^2(w, v) \le \frac{6 \sqrt 2}{\sqrt t} \Big\}\Bigr) \ge 1 - 2n\exp(-t/60) \, .
    \end{align}

    Combining \eqref{eq: usage of Cauchy inequality} with a simple union bound, we have
    \begin{align}
        \mathbb{P}\Bigl(\lVert M_t \rVert_\infty \le \frac{10}{\sqrt t}\Bigr) \ge 1 - 4n \exp(-t/60) \, ,
    \end{align}

    which confirms Corollary \ref{cor: estimate for uniform norm}. 
    
\end{proof}

In conclusion, Corollary \ref{cor: estimate for uniform norm} shows that $ \lVert M_t \rVert = O(t^{-1/2})$ with high probablity for $t \in (0, n^2]$. We now turn our attention to the behavior of the random matrix $M_t$ for large $t$. Lemma \ref{lem: estimate for long time decay of m_t} below shows that $\lVert M^*_t \rVert_\infty = O(\exp(-\frac{t}{8n^2}))$ with high probability. 


\begin{Lemma} \label{lem: estimate of spectral gap}
    For any finite connected graph $G = (V, E)$ with $\lvert V \rvert = n$, and any function $g: V \to \mathbb{R}$ with $\sum_{v \in V} g(v) = 0$, we have
    \begin{align}
        \sum_{\{v, w\} \in E} (g(v) - g(w))^2 \ge \frac{1}{n} (\mathrm{osc} \ g)^2 \ge \frac{1}{n^2} \sum_{v \in V} g(v)^2 \, .
    \end{align}
\end{Lemma}
\begin{proof}
    Suppose $g(u_0) = \sup_{v \in V} m(v)$ and $g(u_1) = \inf_{v \in V} m(v)$. Since $G$ is connected, we can find a simple path from $u_0$ to $u_1$, defined as $\gamma = (v_0, v_1, \cdots, v_\ell)$, where $v_0 = u_0$ and $v_\ell = u_1$. Since $\sum_{v \in V} g(v) = 0$, we have $g(u_0) \ge 0$ and $g(u_1) \le 0$. 

    By the Cauchy-Schwarz inequality, we have
    \begin{align}
        \sum_{\{v, w\} \in E} (g(v) - g(w))^2 &\ge \sum_{i = 1}^\ell (m(v_i) - m(v_{i - 1}))^2 \nonumber \\ &\ge \frac{1}\ell(g(u_0) - g(u_1))^2 \nonumber \\ &\ge \frac{1}{n}(g(u_0) - g(u_1))^2 = \frac{1}{n} (\mathrm{osc} \ g)^2\, .
    \end{align}

    Moreover, the square sum of $g$ has a simple estimate as follows
    \begin{align}
        \sum_{v \in V} g(v)^2 \le n \max(g(u_0), -g(u_1))^2 \le n (\mathrm{osc} \ g)^2 \, .
    \end{align}

    Combining these two inequalities, we conclude Lemma \ref{lem: estimate of spectral gap}. 
\end{proof}

\begin{Lemma} \label{lem: estimate for long time decay of fragmentation}
    For any finite, connected graph $G = (V, E)$ with $\lvert V \rvert = n$, any initial mass distribution $\{\mu_0(v)\}_{v \in V}$, and any $t \ge 0$, the fragmentation process $\{\mu_t(v)\}_{t \ge 0, v \in V}$ satisfies $$\mathbb{P}\Bigl(Q^*_t \ge \exp\Bigl(-\frac{t}{4n^2}\Bigr)\Bigr) \le \exp\Bigl(-\frac{t}{4n^2}\Bigr) \, ,$$
    where $Q^*_t := \sum_{v \in V} (\mu_t(v) - 1/n)^2$. 
\end{Lemma}
\begin{proof}
    We define $\mu^*_t(v) := \mu_t(v) - \frac{1}{n}$ and $X_t := \sum_{\{v, w\} \in E} (\mu^*_t(v) - \mu^*_t(w))^2$. The drift of $Q^*_t$ can be represented by
    \begin{align} \label{eq: calculation of drift}
        D(Q_t^*) &= \lim_{h \to 0^+} \mathbb{E}\Bigl(\frac{Q^*_{t + h} - Q^*_t}{h} \Big | Q^*_t\Bigr) \nonumber \\ &= \sum_{\{v, w\} \in E} 2\Bigl(\frac{m^*_t(v) + m^*_t(w)}{2}\Bigr)^2 - m^*_t(v)^2 - m^*_t(w)^2 =-\frac{1}{2} X_t \, .
    \end{align}

    By Lemma \ref{lem: estimate of spectral gap}, we have $-D(Q^*_t) = \frac{1}{2} X_t \ge \frac{1}{2n^2} Q^*_t$. Let $Y_t := Q^*_t \exp(\frac{t}{2n^2})$. Then we have $D(Y_t) \le 0$. Hence, $\mathbb{E}(Y_t)$ is non-increasing for $t \ge 0$. This implies that
    \begin{align}
        \mathbb{E}(Q^*_t) &= \exp\Bigl(-\frac{t}{2n^2}\Bigr) \mathbb{E}(Y_t) \nonumber \\ &\le \exp\Bigl(-\frac{t}{2n^2}\Bigr) \mathbb{E}(Y_0) = \exp\Bigl(-\frac{t}{2n^2}\Bigr) \Bigl(\sum_{v \in V} \mu_t(v)^2 - \frac{1}{n} \Bigr) \le \exp\Bigl(-\frac{t}{2n^2}\Bigr) \, .
    \end{align}

    Hence, by Markov's inequality, we have $$ \mathbb{P}\Bigl(Q^*_t \ge \exp\Bigl(-\frac{t}{4n^2}\Bigr)\Bigr) \le \exp\left(-\frac{t}{4n^2}\right)\, .$$
\end{proof}

\begin{Lemma} \label{lem: estimate for long time decay of m_t}
    For any finite, connected graph $G = (V, E)$ with $\lvert V \rvert = n \ge 2$ and $t \ge 0$, we have
    \begin{align}
        \mathbb{P}\Bigl(\lVert M^*_t \rVert_\infty \le \exp\Bigl(-\frac{t}{8n^2}\Bigr)\Bigr) \ge 1 - n\exp\Bigl(-\frac{t}{4n^2}\Bigr) \, ,
    \end{align}
    where we define $M^*_t := M_t - \frac{1}{n}J$.
\end{Lemma}

\begin{proof}
    Fix $v \in V$. Note that $M_t(v, \cdot)$ is a fragmentation process with initial mass distribution $\delta_w$. By Lemma \ref{lem: estimate for long time decay of fragmentation}, we have
    \begin{align}
        \mathbb{P}\Bigl(\sum_{w \in V} M^*_t(v, w)^2 \ge \exp\Bigl(-\frac{t}{4n^2}\Bigr)\Bigr) \le \exp\Bigl(-\frac{t}{4n^2}\Bigr) \, .
    \end{align}
    
    Applying a simple union bound, we have
    \begin{align}
        \mathbb{P}\Bigl(\forall v \in V, \sum_{w \in V} M^*_t(v, w)^2 \le \exp\Bigl(-\frac{t}{4n^2}\Bigr)\Bigr) \ge 1 - n \exp\Bigl(-\frac{t}{4n^2}\Bigr) \, .
    \end{align}

    Since $M^*_t(v, w) \le \Bigl(\sum_{w \in V} M^*_t(v, w)^2\bigr)^{1/2}$ for any $v, w \in V$, we have
    \begin{align}
        \mathbb{P}\Bigl(\lVert M^*_t \rVert_\infty \le \exp\Bigl(-\frac{t}{8n^2}\Bigr)\Bigr) \ge 1 - n\exp\Bigl(-\frac{t}{4n^2}\Bigr) \, ,
    \end{align}

    which confirms Lemma \ref{lem: estimate for long time decay of m_t}. 
\end{proof}

\section{Expected $\epsilon$-consensus time for deterministic initial profile}


\begin{Def}
    Given a finite, connected graph $G = (V, E)$ and any (deterministic) initial profile $f: V \to \mathbb{R}$, we define $T_\epsilon(f)$ to be the expectation of $\epsilon$-consensus time given initial profile $f_0 = f$. 
\end{Def}

\begin{Prop} \label{prop: subadditiviy proposition}
    For any finite, connected graph $G = (V, E)$, any two initial profiles 
    $f_1, f_2:V \to \mathbb{R}$ and any two constants $\epsilon_1, \epsilon_2 > 0$, we have
    \begin{align}
        T_\epsilon(f) \le T_{\epsilon_1}(f_1) + T_{\epsilon_2}(f_2) \, ,
    \end{align}
    where we define $f := f_1 + f_2$ and $\epsilon := \epsilon_1 + \epsilon_2$. 
\end{Prop}

\begin{proof}
    We couple two edge-averaging processes with two different initial profiles such that they share the same Poisson clock. In detail, we denote the two edge-averaging processes by $\{f_{t, 1}(v) \}_{t \ge 0, v \in V}$ and $\{f_{t, 2}(v) \}_{t \ge 0, v \in V}$ When an edge $\{v, w \} \in E$ rings at time $t$, we update $f_{t, 1}$ and $f_{t, 2}$ simultaneously, namely that $f_{t, 1}(v) = f_{t, 1}(w) := \frac{1}{2}\left(f_{t^-, 1}(v) + f_{t^-, 1}(w) \right)$ and $f_{t, 2}(v) = f_{t, 2}(w) := \frac{1}{2}\left(f_{t^-, 2}(v) + f_{t^-, 2}(w) \right)$, while keeping the other value fixed. It is not hard to see that $g_{t} := f_{t, 1} + f_{t, 2}$ is an edge-averaging process with initial profile $f = f_1 + f_2$. 

    We define $\epsilon_i$-consensus time for two edge-averaging processes by  $$\tau_i := \inf \{t \ge 0: \mathrm{osc} \ f_{t, i} \le \epsilon_i \} \, .$$

    Moreover, we denote the $\epsilon$-consensus time for $\{g_t(v)\}_{t \ge 0, v \in V}$ by $$ \tau := \inf \{t \ge 0: \mathrm{osc} \ g_t \le \epsilon \}\, .$$

    Note that $\tau \le \max(\tau_1, \tau_2) \le \tau_1 + \tau_2$. Hence, we have
    \begin{align}
        T_{\epsilon}(f) = \mathbb{E}(\tau) \le \mathbb{E}(\tau_1 + \tau_2) = T_{\epsilon_1}(f_1) + T_{\epsilon_2}(f_2) \, .
    \end{align}
\end{proof}

For a general initial profile $f$, a rough estimate of $T_\epsilon(f)$ with respect to $n$, $\epsilon$, and the $\ell^1$ norm of the function $f$ is established in Lemma \ref{lem: estimate of T by ell^1}. Here the $\ell^1$ norm of any function $f: V \to \mathbb{R}$ is simply defined by $$\lVert f \rVert_1 := \sum_{v \in V} \lvert f(v) \rvert \, .$$

\begin{Lemma} \label{lem: estimate of T by ell^1}
    There exist constants $C > 0$ such that for any finite connected graph $G = (V, E)$ with $\lvert V \rvert = n \ge 2$, any non-zero function $f: V \to \mathbb{R}$, and any $\epsilon > 0$, we have
    \begin{align}
        T_\epsilon(f) \le \begin{cases}
            C\lVert f \rVert_1^2/\epsilon^2 & \epsilon \ge \frac{1}{n} \lVert f \rVert_1 \, , \\ C n^2 \left(\log(\frac{1}{n \epsilon}\lVert f \rVert_1) + 1\right)& \epsilon < \frac{1}{n} \lVert f \rVert_1 \, .
        \end{cases}
    \end{align}
\end{Lemma}
\begin{proof}
    Let $f^+ := \max(f, 0)$ and $f^{-} := -\min(f, 0)$. By Proposition \ref{prop: subadditiviy proposition}, we have $T_{\epsilon}(f) \le T_{\epsilon/2}(f^+) + T_{\epsilon/2}(f^-)$. 

    Suppose that for any $\epsilon > 0$ and any non-negative function $f: V \to \mathbb{R}$, we have
    \begin{align}
        T_\epsilon(f) \le \begin{cases}
            C_1 \lVert f \rVert_1^2/\epsilon^2 & \epsilon \ge \frac{1}{n} \lVert f \rVert_1 \, , \\ C_1 n^2 \left(\log(\frac{1}{n \epsilon} \lVert f \rVert_1) + 1 \right)& \epsilon < \frac{1}{n} \lVert f \rVert_1 \, .
        \end{cases}
    \end{align}

    Thus, for general non-zero function $f: V \to \mathbb{R}$, we have
    \begin{align} \label{eq: split into positive and negative part}
        T_\epsilon(f) \le T_{\epsilon/2}(f^+) + T_{\epsilon/2}(f^-) \le \begin{cases}
            4C_1 \lVert f \rVert_1^2/\epsilon^2 & \epsilon \ge \frac{2}{n} \lVert f \rVert_1 \, , \\ C_1n^2 \left(\log(\frac{2}{n \epsilon} \lVert f \rVert_1)  + 1 \right)& \epsilon < \frac{2}{n} \lVert f \rVert_1 \, .
        \end{cases} 
    \end{align}

    It is easy to check that the right hand side of \eqref{eq: split into positive and negative part} is no more than $$ \begin{cases} 4(\log 2 + 1) C_1 \lVert f \rVert_1^2/\epsilon^2 & \epsilon \ge \frac{1}{n} \lVert f \rVert_1 \, , \\ 4(\log 2 + 1) C_1 n^2 \left( \log(\frac{1}{n \epsilon} \lVert f \rVert_1) + 1\right) & \epsilon < \frac{1}{n} \lVert f \rVert_1 \, .\end{cases}$$
    
    This implies that the lemma also holds for general cases, when choosing a larger constant. Without loss of generality, we may assume that $f \ge 0$ from now on. We can also assume that $\lVert f \rVert_1 = \sum_{v \in V} f(v) = 1$. Otherwise, we can replace $f$ by $f/\lVert f \rVert_1$ and $\epsilon$ by $\epsilon/ \lVert f \rVert_1$. Hence, $\{f_t(v)\}_{t \ge 0, v \in V}$ is a fragmentation process. 

    We define $f_t^*(v) := f_t(v) - \frac{1}{n}$ and $Q_t^* := \sum_{v \in V} (f_t(v) - \frac{1}{n})^2$. 

    For any $k \in \mathbb{N}$, we define $m_k := 2^{-k}$ and the stopping time $T_k := \tau_{m_k} = \inf \{t \ge 0: \mathrm{osc} \ f_t \le m_k\} \, .$

    Let us first fix $k \in \mathbb{N}^*$. We claim that  $k \in \mathbb{N}^*$, the sequence $\{ M_{t, k}\}_{t \ge 0}$ defined as $$M_{t, k} := Q_t^* + \frac{1}{16}\max\Bigl(\frac{1}{n}m_k^2, m_k^3\Bigr) \min(t, T_k)$$ is a supermartingale. 

    By \eqref{eq: calculation of drift}, the drift of $Q_t^*$ is 
    \begin{align}
        D(Q_t^*) = -\frac{1}{2} \sum_{\{ v, w \} \in E} (f_t(v) - f_t(w))^2 \, .
    \end{align}

    It suffices to show that for any $t < T_k$ we have $$-D(Q_t^*) \ge \frac{1}{16} \max\Bigl(\frac{1}{n}m_k^2, m_k^3\Bigr) \, .$$ 
    
    We will deduce that $-D(Q_t^*) \ge \frac{1}{16}(\mathrm{osc} \ f_t)^3$ almost surely for $t \le T_k$ and $m_k \ge \frac{2}{n}$. Part of the argument is similar to Lemma 2.1 in \cite{gantert2024averagingprocessinfinitegraphs}. 
    
    By Lemma \ref{lem: estimate of spectral gap}, we know that $-D(Q_t^*) \ge \frac{1}{n}(\mathrm{osc} \ f_t)^2$. Hence we have 
    \begin{align} \label{eq: lower bound1 for drift}
        -D(Q_t^*) \ge \frac{1}{n} m_k^2 \quad \forall t \le T_k \, .
    \end{align}

    We fix $t \ge 0$, $k \in \mathbb{N}^*$ and the function $f_t$. Let $u \in V$ be a vertex such that $f_t(u) = \max_{v \in V} f_t(v) \ge m_k \ge \frac{2}{n}$. Since $\sum_{v \in V} f_t(v) = 1$, there is a vertex $w \in V$ for which $f_t(w) \le \frac{1}{n} \le \frac{m_k}{2}$. Let $\gamma = (v_0, v_1, \cdots, v_\ell)$ be a shortest path from $v_0 := u$ to the set $\{w \in V: f_t(w) \le \frac{m_k}{2}\}$. By the Cauchy-Schwarz inequality, we have
    \begin{align}
        \frac{m_k^2}{4} \le \left(f_t(v_0) - f_t(v_\ell)\right)^2 \le \ell \sum_{i = 1}^\ell (f_t(v_i) - f_t(v_{i - 1}))^2 \, . 
    \end{align}

    Note that $$\sum_{v \in V} f_t(v) \ge \sum_{i = 0}^{\ell - 1} f_t(v_\ell) \ge \frac{m_k}{2} \ell \, , $$

    which implies that $\ell \le \frac{2}{m_k}$. Hence we have
    \begin{align} \label{eq: lower bound2 for drift}
        -D(Q_t^*) \ge \frac{1}{2} \sum_{i = 1}^\ell (f_t(v_i) - f_t(v_{i - 1}))^2 \ge \frac{m_k^2}{8 \ell} \ge \frac{m_k^3}{16} \quad \forall t \le T_k, m_k \ge \frac{2}{n} \, .
    \end{align}

    The combination of \eqref{eq: lower bound1 for drift} and \eqref{eq: lower bound2 for drift} implies that $-D(Q_t^*) \ge \frac{1}{16} \max(\frac{1}{n} m_k^2, m_k^3)$. Hence we know that $M_{t, k}$ is a supermartingale. 

    By Optimal stopping theorem, we have $\mathbb{E}(M_{T_k, k}) \le \mathbb{E}(M_{T_{k - 1}, k})$, which implies that
    \begin{align} \label{eq: technical ineq}
        \frac{1}{16} \max\Bigl(\frac{1}{n}m_k^2, m_k^3\Bigr) \mathbb{E}(T_k - T_{k - 1}) \le \mathbb{E}(Q_{T_{k - 1}}^*) \, .
    \end{align}

    Since $\sum_{v \in V} f_t(v) = 1$, we have $\lVert f_t^* \rVert_\infty \le \mathrm{osc} \  f_t$, where we recall that $f_t^* := f_t - \frac{1}{n}$. We also have $\lVert f_t^* \rVert_1 \le \lVert f_t(v) \rVert_1 + 1 = 2$ and $\lVert f_t^* \rVert_1 \le n \lVert f_t^* \rVert_\infty$ This implies that
    \begin{align}
        Q_t^* \le \lVert f_t^* \rVert_\infty \cdot \lVert f_t^* \rVert_1 \le (\mathrm{osc} \ f_t) \min\Bigl(n \ \mathrm{osc} \ f_t, 2\Bigr) \, .
    \end{align}

    This implies that $Q_{T_{k - 1}}^* \le m_{k - 1} \min\Bigl(m_{k - 1}n, 2\Bigr) \le 4m_k n\min\Bigl(m_k, \frac{1}{n}\Bigr)$ almost surely. By \eqref{eq: technical ineq}, we have
    \begin{align} \label{eq: bound for gaps of T_k} 
        \mathbb{E}(T_k - T_{k - 1}) \le \frac{64 m_k n \min(m_k, \frac{1}{n})}{\max\left(\frac{1}{n}m_k^2, m_k^3\right)} = 64 \min\Bigl(\frac{1}{m_k}, n\Bigr)^2 \, .
    \end{align}

    Let us bound $\mathbb{E}(\tau_\epsilon)$. Let $K := \lceil \log_2(1/\epsilon)\rceil$. Then we have
    \begin{align}
        \mathbb{E}(\tau_\epsilon) \le \sum_{k = 1}^{K} \mathbb{E}(T_k - T_{k - 1}) \, .
    \end{align}

    Recall that $m_k = \frac{1}{2^k}$. By \eqref{eq: bound for gaps of T_k}, there exists a constant $C_1 > 0$ such that 
    \begin{align}
        \mathbb{E}(\tau_\epsilon) \le C_1\sum_{k = 1}^{K} \min(2^k, n)^2 \, .
    \end{align}

    Since $K = \lceil \log_2(1/\epsilon) \rceil$, there exists another constant $C_2 > 0$ such that
    \begin{align}
        \mathbb{E}(\tau_\epsilon) \le \begin{cases}
            C_2\lVert f \rVert_1^2/\epsilon^2 & \epsilon \ge \frac{1}{n} \lVert f \rVert_1 \, , \\ C_2 n^2 \left(\log(\frac{1}{n \epsilon}\lVert f \rVert_1) + 1\right)& \epsilon < \frac{1}{n} \lVert f \rVert_1 \, .
        \end{cases}
    \end{align}
    
    This concludes Lemma \ref{lem: estimate of T by ell^1}. 
    
\end{proof}

Theorem \ref{thm: upper bound for L^1} can be easily deduced by Lemma \ref{lem: estimate of T by ell^1}. 

\begin{proof}[Proof of Theorem \ref{thm: upper bound for L^1}]
    By Lemma \ref{lem: estimate of T by ell^1}, for some constant $C > 0$, we have
    \begin{align}
        \mathbb{E}(\tau_\epsilon) = \mathbb{E}(T_\epsilon(f_0)) \le Cn^2 \mathbb{E}\Bigl(\log \max\Big(\frac{\lVert f \rVert_1}{n \epsilon}, 2\Big)\Bigr) \, .
    \end{align}

    Since $\mathbb{E}(\lvert f_0(v) \rvert) \le 1$ for any $v \in V$, we have $\mathbb{E}(\lVert f \rVert_1) \le n$. 
    
    Note that $x \to \log \max(\frac{x}{n \epsilon}, 2)$ is convex. By Jensen's inequality, we have
    \begin{align}
        \mathbb{E}\Bigl(\log \max\Big(\frac{\lVert f \rVert_1}{n \epsilon}, 2\Big)\Bigr) \le \log \max\Bigl(\frac{\mathbb{E}(\lVert f \rVert_1)}{n \epsilon}, 2\Bigr) \le \log(2/\epsilon)\, ,
    \end{align}

    which concludes Theorem \ref{thm: upper bound for L^1}. 
\end{proof}

\section{Proof of Theorem \ref{thm: upper bound for p > 1}}

\begin{Def}
    For any $v \in V$, we define $\mathcal{F}_v$ to be the $\sigma$-field generated by the random variable $f_0(v)$. We also define $\mathcal{F}$ to be the $\sigma$-field generated by the initial profile $\{f_0(v)\}_{v \in V}$, and $\mathcal{G}$ to be the $\sigma$-field generated by the Poisson clock on each edge. 
\end{Def}

We will commonly use the fact that $\mathcal{F}$ and $\mathcal{G}$ are independent $\sigma$-algebra. 

\begin{Lemma} \label{lem: estimate for small part by Bernstein}
    There exists a constant $C > 0$ such that for any finite connected graph $G = (V, E)$ with $\lvert V \rvert = n \ge 2$, any $\epsilon \in (0, 1]$, $B \in [1, n\epsilon/C \log n]$ and $V \in [1, n\epsilon/C \log n]$, the following statement holds: for any random variable $X$ with $\lvert X \rvert \le B$ a.s.\ and $\mathbb{E}(X^2) \le V$, and any i.i.d.\ initial profile $f: V \to \mathbb{R}$ with $f(v) \sim X$ satisfies
    \begin{align}
        \mathbb{E}\Bigl(T_\epsilon(f)\Bigr) \le C \epsilon^{-4} (\epsilon B + V)^2 \log^2(n/\epsilon) \, .
    \end{align}
\end{Lemma}

\begin{proof} 
    Without loss of generality, We may assume that $\mathbb{E}(X) = 0$. Otherwise, we replace $X$ by $X - \mathbb{E}(X)$ and $B$ by $2B$. 
    
    By Corollary \ref{cor: estimate for uniform norm} and Lemma \ref{lem: estimate for long time decay of m_t}, there exists a constant $C_1 > 0$ such that the following two facts hold
    \begin{enumerate}
        \item  $\mathbb{P}\Bigl(\lVert M_t \rVert_\infty \le \frac{C_1}{\sqrt t}\Bigr) \ge 1 - 4n \exp(-t/C_1)$ for $t \in [1, n^2]$.  
        \item $\mathbb{P}\Bigl(\lVert M^*_t \rVert_\infty \le \exp\Bigl(-\frac{t}{C_1n^2}\Bigr)\Bigr) \ge 1 - n \exp\Bigl(-\frac{t}{C_1n^2}\Bigr)$ for $t \in (n^2, \infty)$. 
    \end{enumerate}
    
    Let $\{f_t(v)\}_{t \ge 0, v \in V}$ be the edge-averaging dynamics with initial profile $f_0 := f$. We recall that $f_t = M_t^T f_0$. Since $\mathcal{F}$ and $\mathcal{G}$ are independent, fixing $v \in V$ and conditioning on $\mathcal{G}$, the random variables $\{ M_t(w, v) f_0(w)\}_{w \in V}$ are independent. 

    Note that $\{M_t(w, v) f_0(w)\}_{w \in V}$ are bounded by $\lVert M_t \rVert_\infty B$. By Bernstein's inequality, we have
    \begin{align} \label{eq: main eq of Bernstein inequality}
        \mathbb{P}\bigl(\lvert f_t(v)\rvert \ge \epsilon \, \mid \mathcal{G}\bigr) &\le 2 \exp\left(-\frac{\epsilon^2}{2 \sum_{w \in V} M_t(w, v)^2 V + \frac{2}{3}\lVert M_t \rVert_\infty B\epsilon}\right) \nonumber \\ &\le 2 \exp\Bigl(-\frac{\epsilon^2}{2 \lVert M_t \rVert_\infty V + \frac{2}{3} \lVert M_t \rVert_\infty B \epsilon}\Bigr) \, ,
    \end{align}
    where second step follows from the fact that $\sum_{w \in V} M_t(w, v) = 1$. 
    
    
    We fix time $t \in [1, n^2]$. A simple union bound shows that
    \begin{align} \label{eq: application of Bernstein}
        \mathbb{P}(\tau_\epsilon > t) &\le \mathbb{P}(\lVert M_t \rVert_\infty \ge C_1 t^{-1/2}) + \sum_{v \in V} \mathbb{P}\Bigl(\lVert M_t \rVert_\infty < C_1 t^{-1/2} \text{ and } \lvert f_t(v) \rvert \ge \epsilon/2\Bigr) \nonumber \\ &\le 
        C_2\Bigl(n \exp\Bigl(-\frac{t}{C_1}\Bigr) + n \exp\Bigl(-\frac{\epsilon^2}{C_2 t^{-1/2} (\epsilon B + V)}\Bigr)\Bigr) \, ,
    \end{align}
    where the second step follows from Fact {\bf (a)} and \eqref{eq: main eq of Bernstein inequality}. 
    
    We assume that $C_3 \epsilon^{-4} (\epsilon B + V)^2 \log^2n \le n^2$. By \eqref{eq: application of Bernstein}, there exists a constant $C_3 > 0$ such that for any $t \in [C_3 \epsilon^{-4}(\epsilon B + V)^2 \log^2 n, n^2]$, we have
    \begin{align}
        \mathbb{P}(\tau_\epsilon > t) \le n^{-3}\exp\Bigl(-\frac{\epsilon^2 \sqrt t}{C_3(\epsilon B + V)}\Bigr) \, .
    \end{align}

    We assume that $C_3 \epsilon^{-4} (\epsilon B + V)^2 \log^2 \le n^2$ from now on. 
    
    By Lemma \ref{lem: estimate for long time decay of m_t}, there exists a constant $C_4 > 1$ such that for any $t > 8C_4 n^2 \log(2nB/\epsilon)$, we have
    \begin{align}
        \mathbb{P}\big(\lVert M^*_t\rVert_\infty \le \frac{\epsilon}{2B}\big) &\ge \mathbb{P}\Bigl( \lVert M^*_t \rVert_\infty \le \exp\Big(-\frac{t}{C_4n^2}\Big)\Bigr) \nonumber \\ &\ge 1 - n\exp\Bigl(-\frac{t}{C_4n^2}\Bigr) \nonumber \\ &\ge 1-n^{-3} \exp\Bigl(-\frac{t}{2C_4n^2}\Bigr)\, .
    \end{align}

    Since $f_t = M_t^T f_0$, we have $\mathrm{osc} \ f_t \le 2B \lVert M^*_t \rVert_\infty$. This implies that
    \begin{align} \label{eq: tail bound for tau epsilon}
        \mathbb{P}(\tau_\epsilon > t) \le n^{-3} \exp\Bigl(-\frac{t}{2C_4n^2}\Bigr) \quad \forall t \in [8C_4n^2 \log(2nB/\epsilon), \infty) \, .
    \end{align}
    
    Then the expected $\epsilon$-consensus time $\mathbb{E}(\tau_\epsilon) = \mathbb{E}(T_\epsilon(f))$ can be bounded by
    \begin{align} \label{eq: detail bound for T epsilon f}
        \mathbb{E}(T_\epsilon(f)) &= \int_0^\infty \mathbb{P}(\tau_\epsilon > t) \mathrm{d} t \nonumber \\ &\le C_3 \epsilon^{-4}(\epsilon M + V)^2 \log^2n + \int_{C_3 \epsilon^{-4}(\epsilon B + V) \log^2n}^{n^2} n^{-3} \exp \Bigl(-\frac{\epsilon^2 \sqrt t}{C_3(\epsilon B + V)}\Bigr) \mathrm{d} t\nonumber \\ &+ \Bigl(8C_4n^2 \log(2nB/\epsilon) - n^2\Bigr) \mathbb{P}(\tau_\epsilon > n^2) + \int_{8C_4n^2 \log(2nB/\epsilon)}^{\infty} n^{-3} \exp\Bigl(-\frac{t}{C_4n^2}\Bigr) \mathrm{d} t \, ,
    \end{align}
    where we use the tail bound of $\tau_\epsilon$ for $[C_3 \epsilon^{-4}(\epsilon B + V)^2 \log^2 n, n^2]$ and $[8C_4 n^2 \log(2nB/\epsilon), \infty)$. 
    
    A simple calculation using \eqref{eq: tail bound for tau epsilon} shows that the second, third and fourth terms of the right hand side of \eqref{eq: detail bound for T epsilon f} have order $O(\epsilon^{-4}(\epsilon B + V)^2 \log^2 (n/\epsilon))$. Hence, there exists a constant $C_5 > 0$ such that 
    \begin{align}
        \mathbb{E}\left(T_\epsilon(f)\right) \le C_5 \epsilon^{-4}(\epsilon M + V)^2 \log^2 n \quad \forall C_3 \epsilon^{-4}(\epsilon M + V)^2 \log^2 n \le n^2 \, .
    \end{align}

    This concludes Lemma \ref{lem: estimate for small part by Bernstein}. 
\end{proof}

Let us go back to the proof of Theorem \ref{thm: upper bound for p > 1}. The strategy to prove Theorem \ref{thm: upper bound for p > 1} is to split $f_0$ into big part and small part, by choosing a proper threshold $B > 0$. The expected $\epsilon$-consensus time of big part can be bounded by Lemma \ref{lem: estimate of T by ell^1}, while the small part can be bounded by Lemma \ref{lem: estimate for small part by Bernstein}. 

Let $B := 2n^{1/p}$ and $V := \max\left(\mathbb{E}(X^2 1_{\lvert X \rvert \le B}), 1 \right)$. We define two random functions $g$ and $h$ by $$ g(v) := f_0(v) 1_{\lvert f_0(v) \rvert \le B}, h(v) := f_0(v) 1_{\lvert f_0(v) \rvert > B}\, .$$

Intuitively, $g$ is the small part of $f_0$ and $h$ is the big part. The expected $\epsilon$-consensus time for $g$ and $h$ can be bounded via different techniques. 

By Lemma \ref{prop: subadditiviy proposition}, we have 
\begin{align} \label{eq: usage of subadditivity}
    \mathbb{E}(\tau_\epsilon) \le \mathbb{E}\left(T_{\epsilon/2}(g)\right) + \mathbb{E}\left(T_{\epsilon/2}(h)\right).
\end{align}

Hence, it suffices to estimate the expectation of $T_{\epsilon/2}(g)$ and $T_{\epsilon/2}(h)$, applying different techniques. 

Since $\mathbb{E}(\lvert X \rvert^p) \le 1$ and $\lvert X \rvert \le B$ a.s.\ , we have the following estimate for $V$
\begin{align} \label{eq: estimate for V}
    V \le \begin{cases}
        B^{2 - p} & p \in (1, 2) \, , \\ 1 & p \in [2, \infty) \, .
    \end{cases}
\end{align}

Hence for both $p \in (1, 2)$ case and $p \in [2, \infty)$ case, we always have $B \le 2n^{1/p}$ and $V \le 2n^{1/p}$. By Lemma \ref{lem: estimate for small part by Bernstein}, there exists a constant $C_1 > 2$ such that
\begin{align} \label{eq: estimate for T f_0, 1}
    \mathbb{E}(T_{\epsilon/2}(g)) \le C_1 \epsilon^{-4}(\epsilon B + V)^2 \log^2(n/\epsilon) \ \forall \epsilon \in [C_1n^{\frac{1}{p} - 1} \log n, 1]\, .
\end{align}

We will assume that $\epsilon \ge C_1 n^{1/p - 1} \log n$ from now on. By Lemma \ref{lem: estimate of T by ell^1}, there exists a constant $C_2 > 1$ such that
\begin{align} \label{eq: estimate1 for T h}
    T_{\epsilon/2}(h) \le \begin{cases}
        C_2 \lVert h \rVert_1^2/\epsilon^2 & \lVert h \rVert_1 \le n \epsilon \, , \\ C_2 n^2 \log(\frac{1}{n \epsilon} \lVert h \rVert_1) & \lVert h \rVert_1 > n \epsilon \, .
    \end{cases}
\end{align}


We define the random variable $S$ by $$S := \lVert h \rVert_1 = \sum_{v \in V} \lvert h(v) \rvert \, .$$ 

In order to give a precise upper bound for $\mathbb{E}(T_{\epsilon/2}(h))$, we first bound $\mathbb{E}(S^2 1_{S \le n \epsilon})$ and $\mathbb{E}(S^p)$.  

The quantity $\mathbb{E}(S^2 1_{S \le n \epsilon})$ can be bounded as follows
\begin{align} \label{eq: estimate for S^2}
    \mathbb{E}(S^2 1_{S \le n \epsilon}) &\le \mathbb{E}\Bigl( \Big(\sum_{v \in V} \lvert f_0(v) \rvert 1_{B < \lvert f_0(v) \rvert \le n \epsilon }\Big)^2\Bigr) \nonumber \\ &\le n^2 \mathbb{E}(\lvert X \rvert 1_{\lvert X \rvert > B})^2 + n \mathbb{E}(X^2 1_{B < \lvert X \rvert \le n \epsilon}) \nonumber \\ &\le n^2B^{2(1 - p)} \mathbb{E}(\lvert X \rvert^p)^2 + n \mathbb{E}(X^2 1_{B < \lvert X \rvert \le n \epsilon}) \nonumber \\ &\le n^2B^{2(1 - p)} + n \mathbb{E}(X^2 1_{B < \lvert X \rvert \le n \epsilon})\, ,
\end{align}

where we use the property that $\mathbb{E}(\lvert X \rvert^p) \le 1$. The property also implies that
\begin{align} \label{eq: estimate for X^2}
    \mathbb{E}(X^2 1_{B < \lvert X \rvert \le n \epsilon}) \le \begin{cases}
        (n \epsilon)^{2 - p} \mathbb{E}(\lvert X \rvert^p) \le (n \epsilon)^{2 - p}& p \in (1, 2) \, , \\ M^{2 - p} \mathbb{E}(\lvert X \rvert^p) \le M^{2 - p} & p \in [2, \infty) \, .
    \end{cases}
\end{align}

In order to estimate $\mathbb{E}(S^p)$, we define random variables $Z$ and $\{Z_v\}_{v \in V}$ as follows
\begin{gather}
    Z_v := 1_{\lvert f_0(v) \rvert > B} \, ,  \\ Z := \sum_{v \in V} Z_v \, .
\end{gather}

Note that we have $$ S^p = \Bigl(\sum_{v \in V} Z_v \lvert f_0(v) \rvert \Bigr)^p \le Z^p \Bigl(\sum_{v} \lvert f_0(v) \rvert^p \Bigr) \, .$$

Also note that for any $v \in V$, $f_0(v)$ and $Z - Z_v$ are independent. Hence, we have
\begin{align*}
    \mathbb{E}(S^p) &\le \mathbb{E}\Bigl(Z^p \sum_{v \in V} \lvert f_0(v) \rvert^p\Bigr) \\ &\le \sum_{v \in V} \mathbb{E}\Bigl(\lvert f_0(v) \rvert^p(Z - Z_v + 1)^p\Bigr) \\ &= \sum_{v \in V} \mathbb{E}(\lvert f_0(v) \rvert^p) \mathbb{E}((Z - Z_v + 1)^p) \\  &\le n \mathbb{E}\left((Z + 1)^p\right) \le n\mathbb{E}\left(2^{pZ}\right) \, .
\end{align*}

By Markov's inequality, we have $\mathbb{P}(Z_v = 1) \le \frac{1}{B^p} = \frac{1}{2^pn}$ for any $v \in V$. Hence we have 
\begin{align}
    \mathbb{E}(2^{pZ}) = \prod_{v \in V} \Bigl(1 + (2^p - 1) \mathbb{P}(Z_v = 1) \Bigr) \le \Bigl(1 + \frac{2^p - 1}{2^pn}\Bigr)^n \le 3\, .
\end{align}

This implies that $\mathbb{E}(S^p) \le 3n$. 

By \eqref{eq: estimate1 for T h}, we obtain
\begin{align} \label{eq: estimate3 for T h}
    \mathbb{E}\left(T_{\epsilon/2}(h)\right) \le C_2 \epsilon^{-2}\mathbb{E}\left(S^2 1_{S \le n \epsilon}\right) + C_2 n^2 \mathbb{E}\Bigl(\log\big(\frac{S}{n \epsilon}\big) 1_{S \ge n \epsilon}\Bigr)\, .
\end{align}

The first term of \eqref{eq: estimate3 for T h} is bounded by \eqref{eq: estimate for S^2}, while the second term can be simply bounded by the $p$-th moment of $S$ as below
\begin{align} \label{eq: bound for second term}
    \mathbb{E}\Bigl(\log\Bigl(\frac{S}{n \epsilon}\bigr) 1_{S \ge n \epsilon}\Bigr) &\le  \mathbb{E}\left(\log\Big(\frac{S^p}{(n \epsilon)^p}\Big) 1_{S \ge n \epsilon}\right) \nonumber \\ &\le \frac{1}{(n \epsilon)^p} \mathbb{E}(S^p) \nonumber \\ &\le 3n^{1 - p} \epsilon^{-p} \, ,
\end{align}
where we use the simple fact that $\log\left(\frac{S^p}{(n \epsilon)^p}\right) \le \frac{S^p}{(n \epsilon)^p}$ in the second step, and $\mathbb{E}(S^p) \le 3n$ in the third step.  




By \eqref{eq: estimate for S^2},  \eqref{eq: estimate3 for T h} and \eqref{eq: bound for second term} , we have
\begin{align} \label{eq: estimate4 for T h}
    \mathbb{E}\left(T_{\epsilon/2}(h)\right) \le C_3\Bigl(n^{2/p} \epsilon^{-2} + n \epsilon^{-2} \mathbb{E}(X^2 1_{B < \lvert X \rvert \le n \epsilon}) + n^{3 - p} \epsilon^{-p} \Bigr) \, ,
\end{align}
where we define $C_3 := 4C_2$. 

Let us provide a final estimate for $\mathbb{E}(\tau_\epsilon)$ by combining \eqref{eq: usage of subadditivity}, \eqref{eq: estimate for T f_0, 1}, and \eqref{eq: estimate4 for T h}. 

{\bf Case 1: $p \in (1, 2)$. }

    By the definition of $B$ and \eqref{eq: estimate for V}, we have $B = 2n^{1/p}$ and $V \le B^{2 - p} \le 2n^{2/p - 1}$. Hence, the inequality \eqref{eq: estimate for T f_0, 1} implies
    \begin{align}
        \mathbb{E}(T_{\epsilon/2}(g)) &\le C_1 \epsilon^{-4}(2\epsilon n^{1/p} + 2n^{2/p - 1})^2 \log^2(n/\epsilon) \nonumber \\ &\le 16C_1 n^{2/p}\epsilon^{-2} \log^2(n/\epsilon) \quad \forall \epsilon \in [C_1 n^{1/p - 1} \log n, 1] \, ,
    \end{align}
    
    where we use the fact that $p \in (1, 2)$ and $\epsilon \ge C_1 n^{1/p - 1} \log n \ge n^{1/p - 1}$. 
    
    Recall that we have $\mathbb{E}(X^2 1_{B < \lvert X \rvert \le n \epsilon}) \le (n \epsilon)^{2 - p}$ by \eqref{eq: estimate for X^2}. By \eqref{eq: estimate4 for T h}, we have
    \begin{align}
        \mathbb{E}\left(T_{\epsilon/2}(h)\right) &\le C_3\Bigl( n^{2/p} \epsilon^{-2} + 2n^{3 -p} \epsilon^{-p}\Bigr) \le 3 C_3 n^{3 - p}\epsilon^{-p} \quad \forall \epsilon \in [C_1n^{1 - 1/p} \log n, 1]\, ,
    \end{align}
    
    where we still use $\epsilon \ge n^{1/p - 1}$ in the second step. 
    
    By \eqref{eq: usage of subadditivity}, we know that there exist constants $C_4, C_5 > 0$ such that
    \begin{align}
        \mathbb{E}(\tau_\epsilon) &\le \mathbb{E}\left(T_{\epsilon/2}(g)\right) + \mathbb{E}\left(T_{\epsilon/2}(h)\right) \nonumber \\ &\le C_4\Bigl(n^{2/p}\epsilon^{-2}\log^2(n/\epsilon) + n^{3 - p}\epsilon^{-p} \Bigr)  \nonumber \\ &\le C_5 n^{3 - p} \epsilon^{-p} \log^2(n/\epsilon) \quad \forall \epsilon \in [C_1n^{1/p - 1} \log n, 1] \, , 
    \end{align}

    which confirms the upper bound in Theorem \ref{thm: upper bound for p > 1}. 
    
{\bf Case 2: $p \in [2, \infty)$. }

    Since $B = 2n^{1/p}$ and $V \le 1$, by \eqref{eq: estimate for T f_0, 1}, we have
    \begin{align}
        \mathbb{E}\left( T_{\epsilon/2}(g)\right) &\le C_1 \epsilon^{-4}(2n^{1/p}\epsilon + 1)^2\log^2(n/\epsilon) \nonumber \\ &\le 8C_1(n^{2/p} \epsilon^{-2} + \epsilon^{-4})\log^2(n/\epsilon) \quad   \forall \epsilon \in [C_1n^{\frac{1}{p} - 1} \log n, 1] \, . 
    \end{align}

    Recall that by \eqref{eq: estimate for X^2}, we have $\mathbb{E}(X^2 1_{B < \lvert X \rvert \le n \epsilon}) \le B^{2 - p} \le 2n^{2/p - 1}$. By \eqref{eq: estimate4 for T h}, we have
    \begin{align}
        \mathbb{E}\left(T_{\epsilon/2}(h)\right) &\le C_3(n^{2/p} \epsilon^{-2} + 2n \epsilon^{-2} n^{2/p - 1} + n^{3 - p} \epsilon^{-p}) \nonumber \\ &\le 4C_3 n^{2/p} \epsilon^{-2} \quad \forall \epsilon \in [C_1 n^{1/p - 1} \log n, 1]\, ,
    \end{align}
    
    where we use the fact that $p \in [2, \infty)$ and $\epsilon \ge C_1 n^{1/p - 1} \log n \ge n^{1/p - 1}$. 
      
    By \eqref{eq: usage of subadditivity}, we know that there exists a constant $C_6 > 0$ such that
    \begin{align}
        \mathbb{E}(\tau_\epsilon) &\le \mathbb{E}\left(T_{\epsilon/2}(g)\right) + \mathbb{E}\left(T_{\epsilon/2}(h)\right) \nonumber \\ &\le 8C_1(n^{2/p} \epsilon^{-2} + \epsilon^{-4}) \log^2(n/\epsilon) + 4C_3 n^{2/p} \epsilon^{-2} \nonumber \\ &\le C_6(n^{2/p} \epsilon^{-2} + \epsilon^{-4}) \log^2(n/\epsilon) \, ,
    \end{align}
    
    which confirms the upper bound in Theorem \ref{thm: upper bound for p > 1}. 

\begin{Remark}
    For $p \in (1, 2)$, the upper bound we can get is slightly tighter than stated in Theorem \ref{thm: upper bound for p > 1}. The actual bound we can obtain is
    \begin{align}
        \mathbb{E}(\tau_\epsilon) = O\Bigl(n^{2/p}\epsilon^{-2} \log^2(n/\epsilon) + n^{3 - p}\epsilon^{-p} \Bigr) \, .
    \end{align}
\end{Remark}

\section{Proof of Lower bounds for expected $\epsilon$-consensus time} \label{sec: lower bounds}

In this section, we will prove Theorem \ref{thm: lower bound}. Theorem \ref{thm: key theorem quoted} and Lemma \ref{lem: key lemma quoted}, which are proved in \cite{elboim2025edgeaveragingprocessgraphsrandom}, are useful in our argument.

\begin{Thm}[\cite{elboim2025edgeaveragingprocessgraphsrandom}, Theorem 5.1] \label{thm: key theorem quoted}
    There exists a constant $c > 0$ such that the following statement holds: for any $\epsilon \in (0, 1/3)$, the consensus time $\tau_\epsilon$ for $n$-cycle graph $G = (V, E)$ and i.i.d.\ initial profile $f_0: V \to \mathbb{R}$ with $\mathbb{P}(f_0(v) = 1) = \mathbb{P}(f_0(v) = -1) = 1/2$ satisfies $$ \mathbb{E}(\tau_\epsilon) \ge c \log^2(n)/\epsilon^4\, .$$
\end{Thm}
    
\begin{Lemma} [\cite{elboim2025edgeaveragingprocessgraphsrandom}, Lemma 5.2]\label{lem: key lemma quoted}
    Let $n \ge 15$ be any positive integer and $G = (V, E)$ be an $n$-cycle. Let $o \in V$ be an arbitrary vertex in $V$. Recall that the random matrix $M_t$ is defined in Definition \ref{def: M_t and M^*_t}. Let $t \ge 120$ be any positive time and $J$ be the set of vertices in $V$ whose distance from $o$ is less than $4 \sqrt t$. We define the event $\Omega$ by $$ \Omega := \Bigl\{ \sum_{v \in J} M_t(o, v) \ge \frac{1}{2}\Bigr\} \, .$$

    Then we always have $\mathbb{P}(\Omega) \ge \frac{4}{5}$. 
\end{Lemma}

In the proof of Theorem \ref{thm: lower bound}, we suppose that $X$ is a symmetric random variable (i.e. $X$ and $-X$ have the same distribution).  Recall that we assume $\{ f_0(v) \}_{v \in V}$ are i.i.d.\ with $f_0(v) \stackrel{d} \sim X$. 

\begin{Lemma} \label{lemma: condition on one variable}
     Let $G = (V, E)$ be an $n$-cycle and $x > 0$ be any positive number. We Suppose that $X$ is a random variable with symmetric, discrete distribution, and the initial profile $\{ f_0(v)\}_{v \in V}$ are i.i.d.\ with $f_0(v) \stackrel{d}\sim X$. Let $\{f_t(v)\}_{t \ge 0, v \in V}$ be the edge-averaging process with initial profile $\{f_0(v)\}_{v \in V}$. Let $o \in V$ be an arbitrary vertex in $V$. There exists a constant $C > 2$ such that the following holds: for any $t \in [120, C^{-1}n^2]$, the following estimate holds 
     \begin{align}
        \mathbb{P}\Bigl(\tau_{\frac{x}{C\sqrt t}} \ge t \text{ and } \lvert f_0(o) \rvert \ge x \, \Big| \, \bigvee_{v \neq o} \mathcal{F}_v \Bigr) \ge \frac{2}{5} \mathbb{P}(\lvert f_0(o) \rvert \ge x )\, ,
     \end{align}

     where we recall that $\mathcal{F}_v$ is the $\sigma$-field generated by the random variable $f_0(v)$, and $\bigvee_{v \neq o} \mathcal{F}_v$ denotes the $\sigma$-field generated by all $\mathcal{F}_v$ with $v \neq o$.  
\end{Lemma}

\begin{proof}
    We fix a time $t \ge 120$. Recall that we have defined $J$ to be the set of vertices in $V$ whose distance from $o$ is less than $4 \sqrt t$ in Lemma \ref{lem: key lemma quoted}. In Lemma \ref{lem: key lemma quoted}, we also define the event $\Omega$ by $$\Omega := \Bigl\{ \sum_{v \in J} M_t(o, v) \ge \frac{1}{2}\Bigr\} \, . $$

    By Lemma \ref{lem: key lemma quoted}, we have $\mathbb{P}(\Omega) \ge \frac{4}{5}$. Let $V_1, V_2 \in V$ be two random variables taking values on $V$ defined as follows 
    \begin{gather*}
        V_1 := \operatorname{argmax}_{v \in V} M_t(o, v) \, , \\ V_2 := \operatorname{argmin}_{v \in V} M_t(o, v) \, .
    \end{gather*}

    In order to avoid ambiguity, we may index $V$ by $\{1, 2, \cdots, n \}$ and choose the vertex that attains maximum or minimum value with smallest index. 
    
    On the event $\Omega$, the following two facts hold
    \begin{enumerate}
        \item $M_t(o, V_1) \ge \frac{1}{18 \sqrt t}$. 
        \item There exists $M_t(o, V_2) \le \frac{1}{n}$. 
    \end{enumerate}

    Fact {\bf (a)} follows from $\sum_{v \in J} M_t(o, v) \ge \frac{1}{2}$ and $\lvert J \rvert \le 9 \sqrt t$, while Fact {\bf (b)} follows from $\sum_{v \in V} M_t(o, v) = 1$. 
    
    We assume that $n \ge 36 \sqrt t$ from now on. Hence, on event $\Omega$, Facts {\bf (a)} and {\bf (b)} imply
    \begin{align} \label{eq: gap of two coefficients}
        M_t(o, V_1) - M_t(o, V_2) \ge \frac{1}{36 \sqrt t} \, .
    \end{align}

    Since $f_t = M_t^T f_0$, we can express $f_t(V_1) - f_t(V_2)$ by 
    \begin{align} \label{eq: express the gap of value}
        f_t(V_1) - f_t(V_2) &= \sum_{w \in V} \Bigl(M_t(w, V_1) - M_t(w, V_2)\Bigr)f_0(w) \nonumber \\ &= \Bigl(M_t(o, V_1) - M_t(o, V_2)\Bigr) f_0(o) + \sum_{w \neq o} \Bigl(M_t(w, V_1) - M_t(w, V_2)\Bigr) f_0(w) \, .
    \end{align}

    We define two random variables
    \begin{gather*}
        Y_1 := \Big \lvert \Bigl(M_t(o, V_1) - M_t(o, V_2)\Bigr) f_0(o) + \sum_{w \neq o} \Bigl(M_t(w, V_1) - M_t(w, V_2)\Bigr) f_0(w) \Big \rvert \, ,\\ Y_2 := \Big \lvert -\Bigl(M_t(o, V_1) - M_t(o, V_2)\Bigr) f_0(o) + \sum_{w \neq o} \Bigl(M_t(w, V_1) - M_t(w, V_2)\Bigr) f_0(w) \Big \rvert \, .
    \end{gather*}

    
    Let $\mathcal{F}^* := \bigvee_{v \neq o} \mathcal{F}_v$. Since $\mathcal{F}$ and $\mathcal{G}$ are independent, and the law of $X$ is symmetric, $Y_1$ and $Y_2$ have the same conditional distribution on any measurable sets in $\mathcal{F^*}$. Hence, we have
    \begin{align} \label{eq: trans to max(Y_1, Y_2)}
        &\mathbb{P}\Bigl(\tau_{\frac{x}{36 \sqrt t}} \ge t \text{ and } \lvert f_0(o) \rvert \ge x \, \Big | \, \mathcal{F}^* \Bigr)\nonumber \\ =& \ \mathbb{P}\Bigl( Y_1 \ge \frac{x}{36 \sqrt t} \text{ and } \lvert f_0(o) \rvert \ge x \, \Big | \, \mathcal{F}^*\Bigr) \nonumber \\ = &\ \frac{1}{2} \Bigl(\mathbb{P}\Bigl( Y_1 \ge \frac{x}{36 \sqrt t} \text{ and } \lvert f_0(o) \rvert \ge x \, \Big | \, \mathcal{F}^*\Bigr) \nonumber \\ +& \mathbb{P}\Bigl( Y_2 \ge \frac{x}{36 \sqrt t} \text{ and } \lvert f_0(o) \rvert \ge x \, \Big | \, \mathcal{F}^*\Bigr)\Bigr) \nonumber \\ \ge& \ \frac{1}{2} \mathbb{P}\Bigl(\max(Y_1, Y_2) \ge \frac{x}{36 \sqrt t} \text{ and } \lvert f_0(o) \rvert \ge x \, \Big | \, \mathcal{F}^*\Bigr) \, .
    \end{align}

    Note that $\max(Y_1, Y_2) \ge \lvert M_t(o, V_1) - M_t(o, V_2) \rvert \cdot \lvert f_0(o) \rvert$. Hence, \eqref{eq: trans to max(Y_1, Y_2)} also  implies that
    \begin{align}
        \mathbb{P}\Bigl(\tau_{\frac{x}{36 \sqrt t}} \ge t \text{ and } \lvert f_0(o) \rvert \ge x \, \Big | \, \mathcal{F}^* \Bigr) &\ge \frac{1}{2} \mathbb{P} \Bigl( \lvert M_t(o, V_1) - M_t(o, V_2) \rvert \ge \frac{1}{36 \sqrt t} \text{ and } \lvert f_0(o) \rvert \ge x \Big | \mathcal{F^*}\Bigr) \nonumber \\ &\ge \frac{1}{2}\mathbb{P}(\Omega) \cdot \mathbb{P}(\lvert X \rvert \ge x) \ge \frac{1}{5} \mathbb{P}(\lvert X \rvert \ge x) \, , 
    \end{align}

    which confirms Lemma \ref{lemma: condition on one variable}. 
\end{proof}

\begin{Coro} \label{cor: key coro}
    Let $G = (V, E)$ be an $n$-cycle and $x > 0$ be any positive number. We suppose that $X$ is a random variable with symmetric, discrete distribution and the initial profile $\{ f_0(v)\}_{v \in V}$ are i.i.d.\ with $f_0(v) \stackrel{d}\sim X$. Let $\{f_t(v)\}_{t \ge 0, v \in V}$ be the edge-averaging process with initial profile $\{f_0(v)\}_{v \in V}$. There exists a constant $C > 2$ such that the following holds: for any $t \in [120, C^{-1}n^2]$, we have
     \begin{align}
        \mathbb{P}\left(\tau_{\frac{x}{C\sqrt t}} \ge t \right) \ge \frac{2}{5} \Bigl(1 - \mathbb{P}(\lvert X \rvert < x)^n\Bigr) \, .
     \end{align}
\end{Coro}
        
\begin{proof}
    We define the random set $I \subseteq V$ as follows $$I := \{v \in V: \lvert f_0(v) \rvert \ge x \} \, .$$

    We will prove that for any $t \in [120, C^{-1}n^2]$ and any (deterministic) nonempty set $\emptyset \neq I_0 \subseteq V$, we have
    \begin{align}
        \mathbb{P}\left(\tau_{\frac{x}{C \sqrt t}} \ge t \, \Big | \, I = I_0\right) \ge \frac{2}{5} \, ,
    \end{align}

    where $C > 0$ is the constant given in Lemma \ref{lemma: condition on one variable}. 

    Fix a nonempty set $I_0 \subseteq V$ and arbitrarily choose $o \in I_0$. 

    By Lemma \ref{lemma: condition on one variable}, we have
    \begin{align}
        \mathbb{P}\Bigl(\tau_{\frac{x}{C \sqrt t}} \ge t \text{ and } \lvert f_0(o) \rvert \ge x \, \Big | \, \mathcal{F}^* \Bigr) \ge \frac{2}{5} \mathbb{P}(\lvert X \rvert \ge x) \, .
    \end{align}

    This implies that for any $I_0 \neq \emptyset$, we have
    \begin{align}
        \mathbb{P}\Bigl(\tau_{\frac{x}{C \sqrt t}} \ge t \, \Big | \, I = I_0\Bigr) \ge \frac{2}{5} \, .
    \end{align}

    Hence, the term $\mathbb{P}(\tau_{\frac{x}{C \sqrt t}} \ge t)$ can be estimated as follows
    
    \begin{align}
        \mathbb{P}(\tau_{\frac{x}{C \sqrt t}} \ge t) &\ge \sum_{\emptyset \neq I_0 \subseteq V} \mathbb{P}\left( \tau_{\frac{x}{C_1 \sqrt t}} \ge t \, \Big | \, I = I_0\right) \mathbb{P}(I = I_0) \nonumber \\ &\ge \frac{2}{5} \sum_{\emptyset \neq I_0 \subseteq V} \mathbb{P}(I = I_0) \nonumber \\ &= \frac{2}{5} \left(1 - \mathbb{P}(I = \emptyset)\right) = \frac{2}{5} \left(1 - \mathbb{P}(\lvert X \rvert < x)^n \right) \, .
    \end{align}
\end{proof}

\begin{proof}[Proof of Theorem \ref{thm: lower bound}]
    Let $C_1 > 2$ be the constant given in Corollary \ref{cor: key coro}. Let $t$ and $x_0$ be two positive numbers. We first require that $t \in [120, C_1^{-1}n^2]$ and $\epsilon = \frac{x_0}{C_1 \sqrt t} $, but the exact values of $t$ and $x_0$ will be assigned later. 

    We take the distribution of $X$ as follows
    \begin{gather}
        \mathbb{P}(X = x) = \begin{cases}
        \frac{1}{2x_0^p} & x = x_0 \, , \\ \frac{1}{2x_0^p} & x = -x_0 \, , \\ 1 - \frac{1}{x_0^p} & x = 0 \, ,\
        \end{cases}
    \end{gather}
    where $x_0 > 1$ is a constant depending on $n$, $p$ and $\epsilon$. We will assign the value of $x_0$ later. 
    
    By Corollary \ref{cor: key coro}, we have
    \begin{align} \label{eq: rough estimate for tail bound of tau epsilon}
        \mathbb{P}(\tau_\epsilon \ge t) &\ge \frac{2}{5}\big( 1 - \mathbb{P}(\lvert X \rvert <  x)^n\big) \nonumber \\ &= \frac{2}{5} \left(1 - (1 - x_0^{-p})^n\right) \nonumber \\ &\ge \frac{2}{5}(1 - \exp(-n x_0^p))\, .
    \end{align}

    Note that the derivative of the function $u \to \frac{1 - e^{-u}}{u}$ is $\frac{u + 1 - e^u}{u^2 e^u} \le 0$, hence is monotone decreasing in $\mathbb{R}$. This implies that for $x_0 \ge n^{1/p}$, we have $1 - \exp(-n x_0^{-p}) \ge (1 - e^{-1}) n x_0^p$. From \eqref{eq: rough estimate for tail bound of tau epsilon} we can obtain
    \begin{align} \label{eq: lower bound1 for tail probability}
        \mathbb{P}(\tau_\epsilon \ge t) \ge \frac{2}{5}\left(1 - \exp(-nx_0^{-p})\right) \ge \frac{1}{5} nx_0^{-p} \ \forall x_0 \ge n^{1/p}\, .
    \end{align}
    
    {\bf Case 1:} $p \in [1, 2)$.
    
        In this case, we take $t := C_1^{-2}n^2$ and $x_0 := C_1 \epsilon \sqrt t = n \epsilon$. For $n \ge 120C_1^2$, we have $t \in [120, C_1^{-1}n^2]$. 
        
        Recall that we assume  $\epsilon \ge n^{\frac{1}{p} - 1}$, hence we have $x_0 = n\epsilon \ge n^{1/p}$. By \eqref{eq: lower bound1 for tail probability}, for $n \ge 120C_1^2$, we have
        \begin{align}
            \mathbb{E}(\tau_\epsilon) \ge t \mathbb{P}(\tau_\epsilon \ge t) \ge \frac{1}{5} n^{3 - p} \epsilon^{-p} \, .
        \end{align}
        
    {\bf Case 2:} $p \in [2, \infty)$.
    
        {\bf Case 2.1:} $\epsilon < \frac{1}{3}$ and $\epsilon < 15C_1 n^{-1/p}$.

            In this case, we take $x_0 = 1$. By Theorem \ref{thm: key theorem quoted}, we have $\mathbb{E}(\tau_\epsilon) \ge c_1\epsilon^{-4} \log^2(n/\epsilon)$ for some constant $c_1 > 0$. Since $n \ge 2$, $\epsilon < 1$ and $\epsilon < 15C_1 n^{-1/p}$, we have $$ \mathbb{E}(\tau_\epsilon) \ge (c_1 \log 2) \epsilon^{-4} \ge \frac{c_1 \log 2}{(30C_1)^2}(n^{2/p}\epsilon^{-2} + \epsilon^{-4}) \, .$$ 
            
        {\bf Case 2.2:} $\epsilon \ge \frac{1}{3}$ or $\epsilon \ge 15C_1 n^{-1/p}$. 
        
        In this case, we take $x_0 := 15C_1n^{1/p}$ and $t := (\frac{x_0}{C_1 \epsilon})^2 = (15 n^{1/p} \epsilon^{-1})^2$.
        
        Since $n \ge 2$ and $\epsilon < 1$, we have $t \ge 15^2 > 120$. If $\epsilon \ge \frac{1}{3}$, then we have $t \le 45^2n^{2/p} \le C_1^{-1}n^2$ for $n \ge 45C_1$. If $\epsilon \ge 15C_1 n^{-1/p}$, then we have $t = (15n^{1/p} \epsilon^{-1})^2 \le C_1^{-1} n^2$ for $n \ge 45C_1$. Hence for  $n \ge 45C_1$, we always have $t \in [120, C_1^{-1}n^2]$ and $x_0 \ge n^{1/p}$. By  \eqref{eq: lower bound1 for tail probability}, we have
        \begin{align} \label{eq: delicated lower bound for case 2.2}
            \mathbb{E}(\tau_\epsilon) \ge t \mathbb{P}(\tau_\epsilon \ge t) \ge \frac{2}{5} n^{2/p} \epsilon^{-2} \, .
        \end{align} 

        Both $\epsilon \ge \frac{1}{3}$ and $\epsilon \ge 15C_1n^{-1/p}$ imply that $\epsilon \ge \frac{1}{3} n^{-1/p}$. Hence \eqref{eq: delicated lower bound for case 2.2} implies
        \begin{align}
            \mathbb{E}(\tau_\epsilon) \ge \frac{2}{5} n^{2/p} \epsilon^{-2} \ge \frac{1}{25}(n^{2/p} \epsilon^{-2} + \epsilon^{-4}) \, .
        \end{align}
\end{proof}

\begin{Remark}
    The $\Omega(\epsilon^{-4} \log^2 n)$ lower bound in Theorem \ref{thm: key theorem quoted} also works under the condition of Theorem \ref{thm: lower bound} and $\epsilon \in (0, 1/3)$. For $\epsilon \in (0, 1/3)$, By comparing the two lower bounds and choosing a better one, we can actually get the following better lower bound for some properly constructed random variable $X$
    \begin{align}
        \mathbb{E}(\tau_\epsilon) \ge \begin{cases}
            C^{-1}(n^{3 - p} \epsilon^{-p} + \epsilon^{-4} \log^2 n)& p \in [1, 2) \, , \\ C^{-1}(n^{2/p} \epsilon^{-2} + \epsilon^{-4} \log^2 n) & p \in [2, \infty) \, .
        \end{cases}
    \end{align}
\end{Remark}

The construction of random variable $X$ in Theorem \ref{thm: lower bound} is dependent of $n$. It is natural to ask what bound we can get if the distribution of $X$ is independent of $n$. Theorem \ref{thm: thm2 of lower bound} gives a slightly weaker lower bound under the strengthened constraint. 

\begin{Thm} \label{thm: thm2 of lower bound}
    Fix $p \in [1, \infty)$ and $p' > p$. There exists a random variable $X$ (independent of $n$) with $\mathbb{E}(X) = 0$ and $\mathbb{E}(\lvert X \rvert^p) \le 1$ such that the following holds: there exists a constant $C = C(p, p') > 0$ (depending only on $p$ and $p'$) such that for any positive integer $n \ge C$ and any $\epsilon \in (n^{1/p' - 1}, C^{-1})$, the consensus time $\tau_\epsilon$ for the $n$-cycle graph $G = (V, E)$, and i.i.d.\ initial profile $\{f_0(v)\}_{v \in V}$ with $f_0(v) \stackrel{d} \sim X$ satisfies 
    \begin{align}
        \mathbb{E}(\tau_\epsilon) \ge \begin{cases}
            C^{-1} n^{3 - p'} \epsilon^{-p'} & p' \in (1, 2) \, , \\ C^{-1}n^{2/p'} \epsilon^{-2} & p' \in [2, \infty)\, .
        \end{cases}
    \end{align}
\end{Thm}
\begin{proof}
    Let $C_1 > 1$ be the constant given in Corollary \ref{cor: key coro}. Let $t$ and $x_0$ be two positive numbers. We require that $t \in [120, C_1^{-1}n^2]$, $x_0 \ge 1$ and $\epsilon = \frac{x_0}{C_1 \sqrt t}$, and the exact values of $t$ and $x_0$ will be assigned later. 

    We take the distribution of $X$ as follows
    \begin{gather*}
        \mathbb{P}(X = m) = \mathbb{P}(X = -m) = \frac{1}{C_2m^{p' + 1}} \ \forall m \in \mathbb{N}^* \, , \\ \mathbb{P}(X = 0) = 1 - \sum_{m \ge 1} \frac{2}{C_2m^{p' + 1}} \, ,
    \end{gather*}

    where $C_2 > 0$ is assigned such that $$ \sum_{m \ge 1} \frac{2m^p}{C_2m^{p' + 1}} \le 1 \, .$$
    
    Since $\sum_{m \ge 1} \frac{1}{m^{p' + 1}}$ and $\sum_{m \ge 1} \frac{1}{m^{p' - p + 1}}$ is convergent, one can find a sufficiently large $C_2 > 0$ to guarantee that (a) and (b) holds. For this constant $C_2 > 0$, we have $\mathbb{E}(\lvert X \rvert^p) = \frac{2}{C_2}\sum_{m \ge 1} m^{p - p' - 1} \le 1$. 

    By Corollary \ref{cor: key coro}, for any $t \in [120, C_1^{-1}n^2]$ and $x_0 \ge 1$, we have
    \begin{align} \label{estimate2 of tail probability}
        \mathbb{P}(\tau_\epsilon \ge t) \ge \frac{2}{5}\big(1 - \mathbb{P}(\lvert X \rvert < x_0)^n\big) \, .
    \end{align}

    There exists a constant $C_3 = C_3(p') > 2$ such that for any $x_0 \ge 1$, we have
    \begin{align}
        \mathbb{P}(\lvert X \rvert < x_0) = 1 - \frac{2}{C_2}\sum_{m \ge x_0} \frac{1}{m^{p' + 1}} \le 1 - \frac{1}{C_3 x_0^{p'}} \, . 
    \end{align}

    The inequality \eqref{estimate2 of tail probability} implies that
    \begin{align}
        \mathbb{P}(\tau_\epsilon \ge t) &\ge \frac{2}{5} \left(1 - (1 - C_3^{-1} x_0^{-p'})^n\right) \nonumber \\ &\ge \frac{2}{5} \Bigl( 1 - \exp(-C_3^{-1} n x_0^{-p'})\Bigr) \, .
    \end{align}

    Note that $C_3 > 1$ and the function $u \to \frac{1 - e^{-u}}{u}$ is monotone decreasing(as shown in the proof of Theorem \ref{thm: lower bound}). Hence we have $1 - \exp(-C_3^{-1} n x_0^{-p'}) \ge C_3^{-1}(1 - e^{-1}) n x_0^{-p'} \ge \frac{1}{3}C_3^{-1} n x_0^{-p'}$ for $n x_0^{-p'} \in [0, 1]$. This implies that for some constant $c_1 > 0$, we have
    \begin{align} \label{eq: lower bound2 for tail probability}
        \mathbb{P}(\tau_\epsilon > t) &\ge \frac{2}{5} \left(1 - \exp(-C_3^{-1} n x_0^{-p'}) \right) \nonumber \\ &\ge c_1n x_0^{-p'} \quad \forall x_0 \ge n^{1/p'} \, .
    \end{align}
    
    {\bf Case 1: $p' \in (1, 2)$}
    
        In this case, we take $t := C_1^{-1}n^2$ and $x_0 := C_1 \epsilon \sqrt t$. For $n \ge 120C_1$, we have $t \ge 120$. For $\epsilon \ge n^{\frac{1}{p'} - 1}$, we have $x_0 \ge n\epsilon \ge n^{1/p'}$. Hence by \eqref{eq: lower bound2 for tail probability}, for $n \ge 120C_1$ and $\epsilon \ge n^{\frac{1}{p'} - 1}$, we have
        \begin{align}
            \mathbb{E}(\tau_\epsilon) \ge t \mathbb{P}(\tau_\epsilon \ge t) \ge c_2n^{3 - p'} \epsilon^{-p'} \, ,
        \end{align}

        where $c_2 > 0$ is a constant depending only on $p$ and $p'$.
        
    {\bf Case 2: $p' \in [2, \infty)$}
    
        In this case, we take $x_0 := n^{1/p'}$ and $t := \left(\frac{x_0}{C_1 \epsilon}\right)^2 = (C_1^{-1} n^{1/p'} \epsilon^{-1})^2$. For $\epsilon \ge n^{1/p' - 1}$, we have $t \le C_1^{-1}n^2$. For $\epsilon \le (15C_1)^{-1}$, we have $t \ge 120n^{2/p'} \ge 120$. Hence by \eqref{eq: lower bound2 for tail probability}, for $\epsilon \in [n^{1/p' - 1}, (15C_1)^{-1}]$, we have
        \begin{align}
            \mathbb{E}(\tau_\epsilon) \ge t \mathbb{P}(\tau_\epsilon \ge t) \ge c_3n^{2/p'} \epsilon^{-2} \, , 
        \end{align}

        where $c_3 > 0$ is also a constant depending only on $p$ and $p'$.  
\end{proof}

\section{Conclusions and Open questions}

We studied the edge-averaging dynamics for general finite connected graphs, and provided an effective upper bound of expected consensus time when $\{ f_0(v)\}_{v \in V}$ are i.i.d.\ with $L^p$ norms bounded by $1$. Moreover, Theorem \ref{thm: lower bound} ensures that the upper bound is tight up to logarithmic factors. The results indicate that there is a phase transition at $p = 2$ for $\mathbb{E}(\tau_\epsilon)$. 

However, sharp graph dependent bounds for $\mathbb{E}(\tau_\epsilon)$ with i.i.d.\ initial profiles in $L^p$ are still unknown. It would be  interesting to establish tighter bounds for specific graphs, including $d$-dimensional lattices, expander graphs, etc. 


Another natural problem is to determine moment dependence of the expected $\epsilon$-consensus time in other opinion-dynamics models when endowed with i.i.d.\ initial opinions. Pertinent examples include vertex-averaging dynamics,  synchronous (as in \cite{degroot1974reaching})  and asynchronous (as in \cite{elboim2022asynchronous}) as well as  the  $\ell^p$-energy
minimization dynamics studied by Amir, Nazarov and Peres   \cite{amir2025convergencerateellpenergyminimization}.   

\section{Acknowledgment}

I am grateful to Yuval Peres for many valuable suggestions. 
\nocite*
\bibliographystyle{plain}
\bibliography{ref_for_finite}

@article{elboim2025edgeaveragingprocessgraphsrandom,
    title={The edge-averaging process on graphs with random initial opinions},
    author={Elboim, Dor and Peres, Yuval and Peretz, Ron},
    journal={Proceedings of the National Academy of Sciences},
    volume={122},
    number={33},
    pages={e2423947122},
    year={2025},
    publisher={National Academy of Sciences}
}

@article{gantert2024averagingprocessinfinitegraphs,
    author = {Gantert, Nina and Vilkas, Timo},
    title = {The averaging process on infinite graphs},
    fjournal = {ALEA. Latin American Journal of Probability and Mathematical Statistics},
    journal = {ALEA, Lat. Am. J. Probab. Math. Stat.},
    issn = {1980-0436},
    volume = {22},
    number = {1},
    pages = {815--823},
    year = {2025},
    language = {English},
    doi = {10.30757/ALEA.v22-32},
    zbMATH = {8062154},
    Zbl = {1569.60155}
}

@mist{gollin2025sharingteagraph,
    author = {Gollin, J. Pascal and Hendrey, Kevin and Huang, Hao and Huynh, Tony and Mohar, Bojan and Oum, Sang-il and Yang, Ningyuan and Yu, Wei-Hsuan and Zhu, Xuding},
    title = {Sharing tea on a graph},
    year = {2025},
    howpublished = {arXiv Preprint, {arXiv}:2405.15353},
    url = {https://arxiv.org/abs/2405.15353},
    arXiv = {arXiv:2405.15353}
}

@article{aldous2012lecture,
    author = {Aldous, David and Lanoue, Daniel},
    title = {A lecture on the averaging process},
    fjournal = {Probability Surveys},
    journal = {Probab. Surv.},
    issn = {1549-5787},
    volume = {9},
    pages = {90--102},
    year = {2012},
    language = {English},
    doi = {10.1214/11-PS184},
    zbMATH = {6050911},
    Zbl = {1245.60088}
}

@article{boyd2006randomized,
    title={Randomized gossip algorithms},
    author={Boyd, Stephen and Ghosh, Arpita and Prabhakar, Balaji and Shah, Devavrat},
    journal={IEEE transactions on information theory},
    volume={52},
    number={6},
    pages={2508--2530},
    year={2006},
    publisher={IEEE}
}

@book{nair2022fundamentals,
    author = {Nair, Jayakrishnan and Wierman, Adam and Zwart, Bert},
    title = {The fundamentals of heavy tails. {Properties}, emergence, and estimation},
    fseries = {Cambridge Series in Statistical and Probabilistic Mathematics},
    series = {Camb. Ser. Stat. Probab. Math.},
    volume = {53},
    isbn = {978-1-316-51173-2; 978-1-00-905373-0},
    year = {2022},
    publisher = {Cambridge University Press},
    language = {English},
    doi = {10.1017/9781009053730},
    zbMATH = {7516544},
    Zbl = {1524.60001}
}

@article{zbMATH07496855,
    author = {Chatterjee, Sourav and Diaconis, Persi and Sly, Allan and Zhang, Lingfu},
    title = {A phase transition for repeated averages},
    fjournal = {The Annals of Probability},
    journal = {Ann. Probab.},
    issn = {0091-1798},
    volume = {50},
    number = {1},
    pages = {1--17},
    year = {2022},
    language = {English},
    doi = {10.1214/21-AOP1526},
    zbMATH = {7496855},
    Zbl = {1485.60069}
}

@misc{arXiv:2603.00705,
    author = {Caputo, Pietro and Quattropani, Matteo and Sau, Federico},
    title = {{$L^2$}-cutoff for the averaging process on random regular graphs},
    year = {2026},
    howpublished = {arXiv Preprint, {arXiv}:2603.00705},
    url = {https://arxiv.org/abs/2603.00705},
    arXiv = {arXiv:2603.00705}
}

@misc{amir2025convergencerateellpenergyminimization,
    title={Convergence rate of $\ell^p$-energy minimization on graphs: sharp polynomial bounds and a phase transition at $p=3$}, 
    author={Gideon Amir and Fedor Nazarov and Yuval Peres},
    year={2025},
    howpublished = {arXiv Preprint, {arXiv}:2508.19411}, 
    eprint={2508.19411},
    archivePrefix={arXiv},
    primaryClass={math.PR},
    url={https://arxiv.org/abs/2508.19411}, 
}

@article{degroot1974reaching,
  title={Reaching a consensus},
  author={DeGroot, Morris H},
  journal={Journal of the American Statistical association},
  volume={69},
  number={345},
  pages={118--121},
  year={1974},
  publisher={Taylor \& Francis}
}

@article {elboim2022asynchronous,
    AUTHOR = {Elboim, Dor and Peres, Yuval and Peretz, Ron},
     TITLE = {The asynchronous {D}e{G}root dynamics},
   JOURNAL = {Random Structures Algorithms},
  FJOURNAL = {Random Structures \& Algorithms},
    VOLUME = {65},
      YEAR = {2024},
    NUMBER = {4},
     PAGES = {857--895},
      ISSN = {1042-9832,1098-2418},
   MRCLASS = {60J10 (05C81)},
  MRNUMBER = {4816412},
MRREVIEWER = {Vivek\ S.\ Borkar},
       DOI = {10.1002/rsa.21248},
       URL = {https://doi.org/10.1002/rsa.21248},
}

\end{document}